\documentclass[10pt,a4paper]{amsart}

\usepackage[utf8]{inputenc}
\usepackage{xcolor}
 
\usepackage{amsthm}
\usepackage{graphicx,amssymb,amsmath,latexsym}
\usepackage{caption}
\usepackage{subcaption}

\theoremstyle{plain}
\newtheorem{theorem}{Theorem}[section]
\newtheorem{lemma}[theorem]{Lemma}
\newtheorem{prop}[theorem]{Proposition}

\theoremstyle{definition}
\newtheorem{definition}[theorem]{Definition}
\newtheorem{remark}[theorem]{Remark}
\newtheorem{Example}[theorem]{Example}

\begin{document}
\title{Twisted virtual braids and twisted links}

\author[K.~Negi]{Komal NEGI}
\address[K.~Negi]{Department of Mathematics, Indian Institute of Technology Ropar, Punjab, India.}
\email{komal.20maz0004@iitrpr.ac.in}

\author[M.~Prabhakar]{Madeti PRABHAKAR}
\address[M.~Prabhakar]{Department of Mathematics, Indian Institute of Technology Ropar, Punjab, India.}
\email{prabhakar@iitrpr.ac.in}

\author[S.~Kamada]{Seiichi KAMADA}
\address[S.~Kamada]{Department of Mathematics, Osaka University, Toyonaka, Osaka 560-0043, Japan}
\email{kamada@math.sci.osaka-u.ac.jp}

\date{}



\maketitle

\begin{abstract}
Twisted knot theory introduced by M. Bourgoin is a generalization of knot theory. It leads us to the notion of twisted virtual braids.  In this paper we show theorems for twisted links corresponding to the Alexander theorem and the Markov theorem in knot theory.  We also provide a group presentation and a reduced group presentation 
of the twisted virtual braid group.


\end{abstract}


\section{Introduction}
M.~O.~Bourgoin \cite{1} introduced twisted knot theory as a generalization of knot theory. 
Twisted link diagrams are link diagrams on $\mathbb{R}^2$ possibly with some crossings called virtual crossings and bars which are short arcs intersecting the arcs of the diagrams. Twisted links are diagrammatically defined as twisted link diagrams modulo isotopies of $\mathbb{R}^2$ and local moves called {\it extended Reidemeister moves} which are  
Reidemeister moves (R1, R2, R3), virtual Reidemeister moves (V1, V2, V3, V4) and twisted moves (T1, T2, T3) depicted in Figure~\ref{tm}.  Twisted links correspond to stable equivalence classes of links in oriented three-manifolds which are orientation I-bundles over closed but not necessarily orientable surfaces.  

Twisted links are analogous to virtual links introduced by L.~H.~Kauffman \cite{2}. Virtual link diagrams are link diagrams on $\mathbb{R}^2$ possibly with some virtual crossings. Virtual links are defined as virtual link diagrams modulo isotopies of $\mathbb{R}^2$ and local moves called {\it generalized Reidemeister moves} which are  
Reidemeister moves (R1, R2, R3) and virtual Reidemeister moves (V1, V2, V3, V4) depicted in Figure~\ref{tm}.  Virtual  links correspond to stable equivalence classes of links in oriented three-manifolds which are orientation I-bundles over closed oriented surfaces.

 \begin{figure}[ht]
  \centering
    \includegraphics[width=12cm,height=7.5cm]{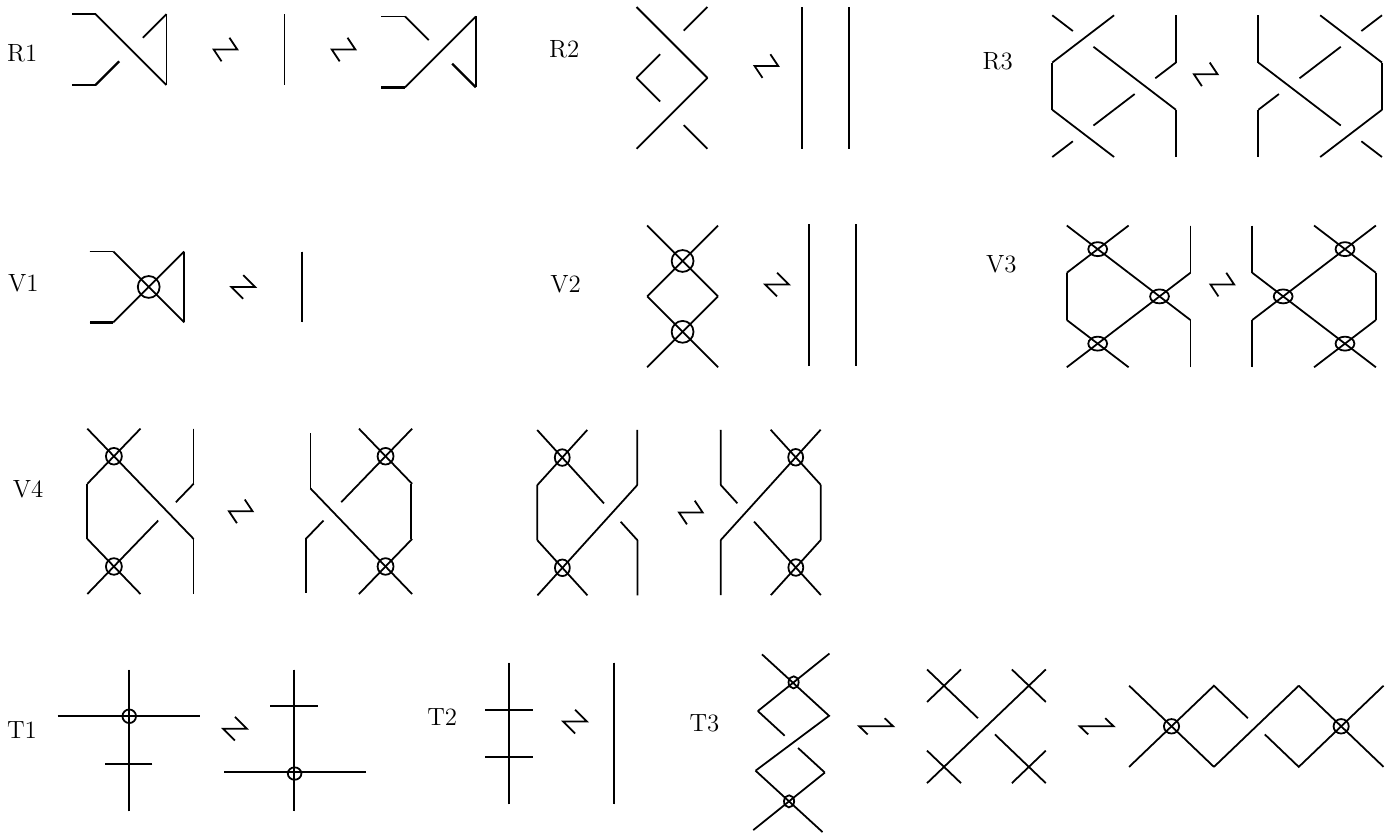}
        \caption{Extended Reidemeister moves.}
        \label{tm}
        \end{figure}

The Alexander theorem states that every link is represented as the closure of a braid, and the Markov theorem states that such a braid is unique modulo certain moves so called Markov moves.  In virtual knot theory, analogous theorems are established in \cite{kl, sk}. 

In this paper we show theorems for twisted links corresponding to the Alexander theorem and the Markov theorem.  We also provide a group presentation and a reduced group presentation 
of the twisted virtual braid group.

This article is organized as follows. 
In Section~\ref{sect:braid}, 
we state the definition of the twisted virtual braid group and provide a group presentation of the group. 
In Section~\ref{sect:Alexander}, the Alexander theorem for twisted links is shown by introducing a method of braiding a given twisted link diagram, which we call the braiding process. 
In Section~\ref{sect:Markov}, we give the statement of the Markov theorem for twisted links and prove it.  In Section~\ref{sect:exchange},  virtual exchange moves are discussed. 
In Section~\ref{sect:reduced}, we give a reduced presentation of the twisted virtual braid group, and concluding remarks.   


\section{The twisted virtual braid group}
\label{sect:braid}

Let $n$ be a positive integer.  

\begin{definition}
A {\it twisted virtual braid diagram} on $n$ strands (or of degree $n$) 
is a union of $n$ smooth or polygonal curves, which are called {\it strands}, in $\mathbb{R}^2$ connecting points $(i,1)$ with points $(q_i,0)$ $(i=1, \dots, n)$, where $(q_1, \ldots, q_n)$ is a permutation of the numbers $(1, \ldots, n)$, such that these curves are monotonic with respect to the second coordinate and intersections of the curves are transverse double points equipped with information as a positive/negative/virtual crossing and 
strings may have {\it bars} by which we mean short arcs intersecting the strings transversely.
See Figure~\ref{exa}, where the five crossings are 
negative, positive, virtual, positive and positive from the top. 
\end{definition}

\begin{figure}[h]
  \centering
    \includegraphics[width=2cm,height=3cm]{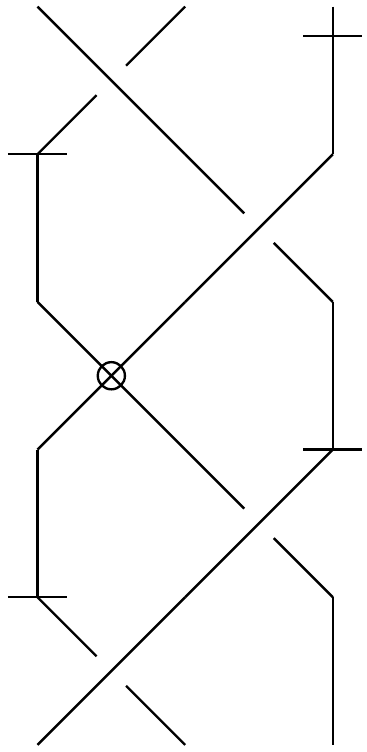}
        \caption{A twisted virtual braid diagram on 3 strands.}
        \label{exa}
        \end{figure}

Here is an alternative definition. 

\begin{definition}
Let $E$ be $[0, n+1] \times [0, 1]$ and let $p_2 : E \to [0, 1]$ be the second factor projection. A {\it twisted virtual braid diagram} of $n$ strands (or of degree $n$) is an immersed 1-manifold $b = a_1 \cup \ldots \cup a_n$ in E, where $a_1, \ldots, a_n$ are embedded arcs, called {\it strands},  
possibly with {\it bars} by which we mean short arcs intersecting the strands transversely, satisfying the following conditions (1)--(5):  
\begin{itemize}
    \item[(1)] $\partial b= \{1, 2, \dots, n\} \times \{0, 1\} \subset E$.
    \item[(2)] For each $i \in \{1, \ldots, n\}$, $p_2|_{a_i}: a_i \to [0, 1]$ is a homeomorphism.  
    \item[(3)] The set of multiple points of the strands consists of transverse double points, which are referred as {\it crossings} of the diagram.  
    \item[(4)] Each crossing is equipped with information of a positive crossing, a negative crossing or a virtual crossing. 
    \item[(5)] Every bar avoids the crossings. 
\end{itemize}
Let $X(b)$ denote the set of crossings of $b$ and the points on the strands where bars intersect with.  
A twisted virtual braid diagram is said to be {\it good} if it satisfies the following condition. 
\begin{itemize}
    \item[(6)] The restriction map $p_2 |_{X(b)}: X(b) \to [0, 1]$ is injective.
\end{itemize}    
\end{definition}

The twisted virtual braid diagram depicted in Figure~\ref{exa} is good. 
On the other hand, the twisted virtual braid diagram depicted in Figure~\ref{exaB} is not good, since there exist a pair of bars lying in $p_2^{-1}(y)$ for some $y \in [0,1]$, or  since there exists a virtual crossing and a bar lying in $p_2^{-1}(y')$ for some $y' \in [0,1]$.

\begin{figure}[h]
  \centering
    \includegraphics[width=2cm,height=3cm]{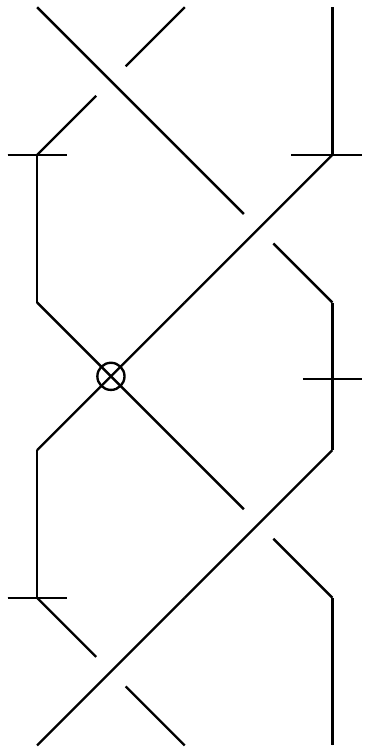}
        \caption{A twisted virtual braid diagram which is not good.}
        \label{exaB}
        \end{figure}

\begin{definition}
Two twisted virtual braid diagrams $b$ and $b'$ of degree $n$ are {\it equivalent} if there is a finite sequence of twisted virtual braid diagrams of degree $n$, say $b_0, b_1, \dots, b_m$, with $b=b_0$ and $b'=b_m$ such that for each $j = 1, \dots, m$, $b_j$ is obtained from $b_{j-1}$ by one of the following: 
\begin{itemize}
    \item An isotopy of $E$ keeping the conditions (1)--(5) of a twisted virtual braid diagram.  
    \item An extended Reidemeister move.  
\end{itemize}   
     A {\it twisted virtual braid} is an equivalence class of twisted virtual braid diagrams.  
\end{definition}

The set of twisted virtual braids forms a group, where the product is defined by the concatenation similar to the braid group such that $b b'$ is $b$ on $b'$ when we draw the braid diagram vertically.
The twisted virtual braid group is denoted by $TVB_n$.  

Let $\sigma_i$, $\sigma_i^{-1}$, $v_i$ $(i=1, \dots, n-1)$ and $\gamma_i$ $(i=1, \dots, n)$ be 
twisted virtual braid diagrams depicted in Figure~\ref{gen}.
Twisted virtual braids represented by them will be also denoted by the same symbols.
The group $TVB_n$ is generated by $\sigma_i$, $v_i$ $(i=1, \dots, n-1)$ and $\gamma_i$ $(i=1, \dots, n)$, which we call {\it standard generators}.  

\begin{figure}[h]
  \centering
    \includegraphics[width=12cm,height=2.5cm]{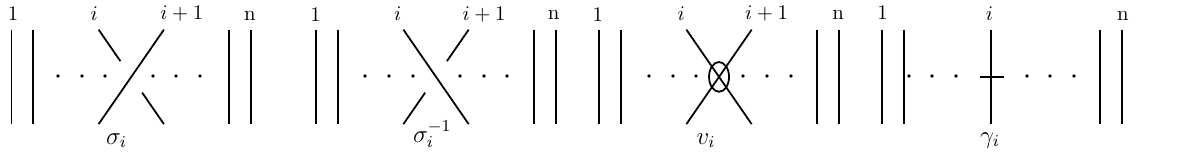}
        \caption{Generators of the group of twisted virtual braids.}
        \label{gen}
        \end{figure}

Figure~\ref{bmoves} shows classical braid moves, corresponding to R2 and R3. Figure~\ref{vbmoves} shows virtual braid moves, corresponding to V2, V3, and V4. 
(There are some other moves corresponding to R3 and V4. However, it is well known that those moves are equivalent to the moves in the figure, cf. \cite{kl}.) 

\begin{figure}[h]
  \centering
    \includegraphics[width=8cm,height=3.5cm]{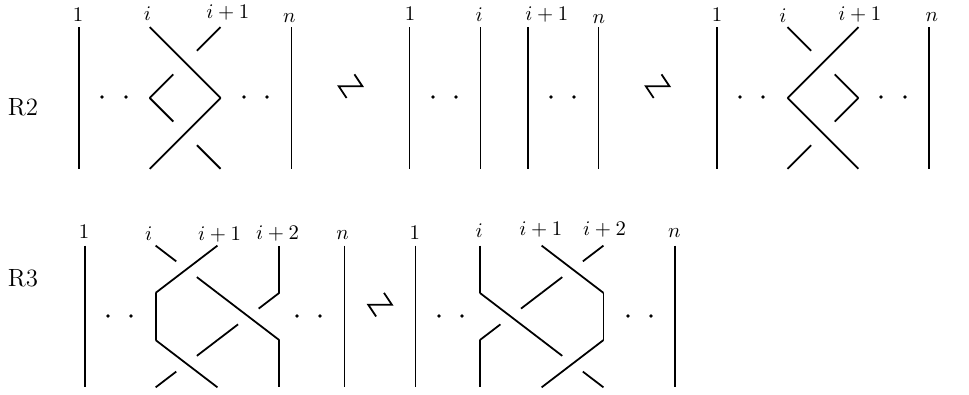}
        \caption{Classical braid moves.}
        \label{bmoves}
        \end{figure}  
        
\begin{figure}[h]
  \centering
    \includegraphics[width=9cm,height=3cm]{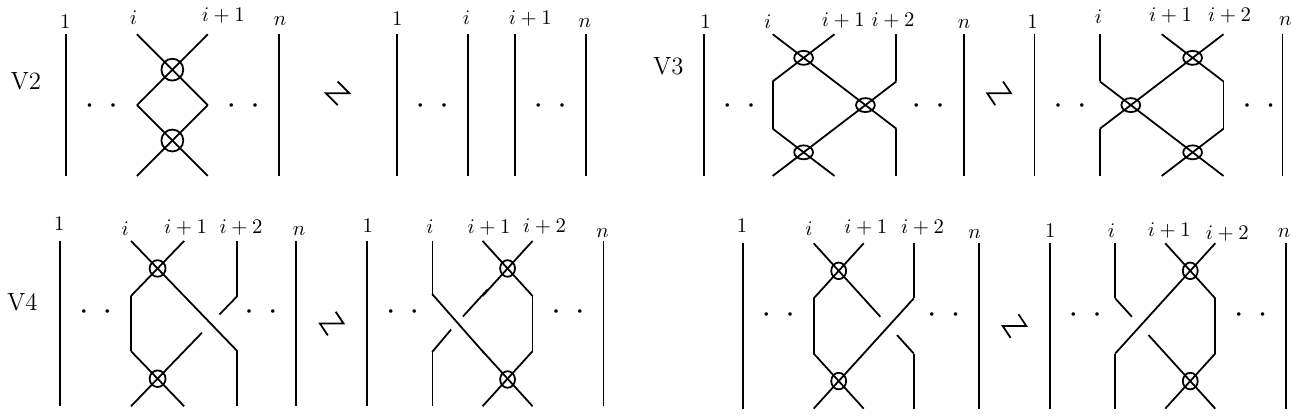}
        \caption{Virtual braid moves.}
        \label{vbmoves}
        \end{figure}  
        
 \begin{figure}[h]
  \centering
    \includegraphics[width=9cm,height=4cm]{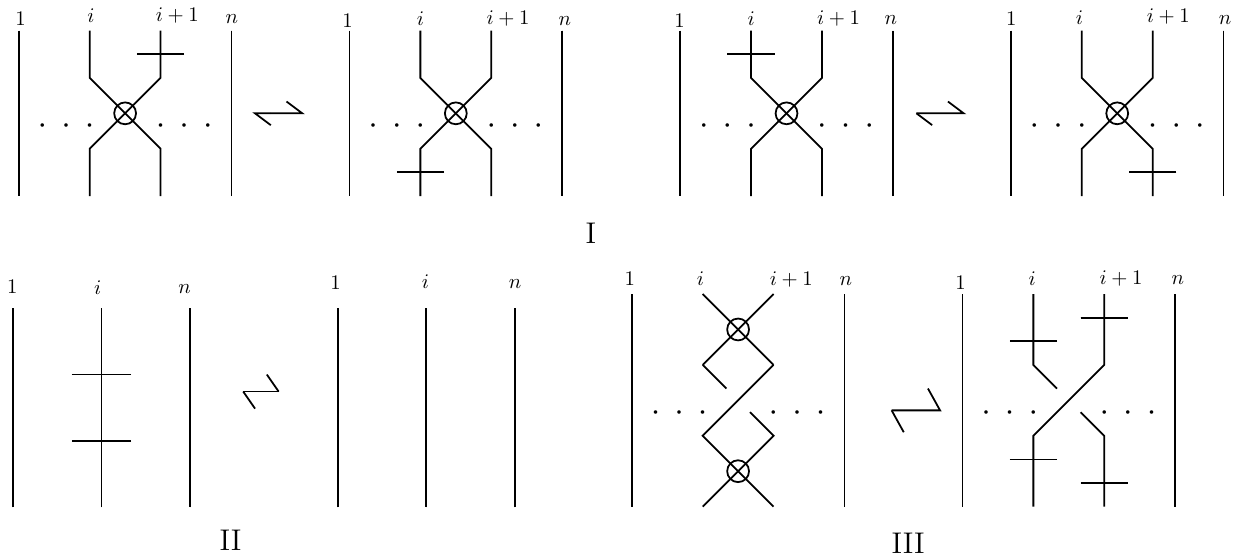}
        \caption{Twisted braid moves.}
        \label{moves}
        \end{figure}  

We call the two moves depicted in the top row of Figure~\ref{moves} {\it twisted braid moves of type ${\rm I}$}, 
and the move on the left of the second row a {\it twisted braid move of type ${\rm II}$}.  
The move on the right of the bottom is called a {\it twisted braid move of type ${\rm III}$} or 
{\it of type ${\rm III(+)}$}.  When we replace the positive crossings with negative ones, it is called 
a {\it twisted braid move of type ${\rm III}$} or {\it of type ${\rm III(-)}$}. 

Braid moves corresponding to extended Reidemeister moves are classical braid moves, virtual braid moves and twisted braid moves.

\begin{theorem}\label{thm:StandardPresentation}
The twisted virtual braid group $TVB_n$ is generated by standard generators, 
$\sigma_i$, $v_i$ $(i=1, \dots, n-1)$ and  $\gamma_i$ $(i=1, \dots, n)$, and  
the following relations are defining relations, where $e$ denotes the identity element:
    \begin{align}
        \sigma_i \sigma_j & = \sigma_j \sigma_i  & \text{ for } & |i-j| > 1; \label{rel-height-ss}\\
        \sigma_i \sigma_{i+1} \sigma_i & = \sigma_{i+1} \sigma_i \sigma_{i+1} & \text{ for } & i=1,\ldots, n-2; \label{rel-sss}\\
        v_i^2 & = e  & \text{ for } & i=1,\ldots, n-1; \label{rel-inverse-v}\\
        v_i v_j & = v_j v_i & \text{ for } & |i-j| > 1 ; \label{rel-height-vv}\\
        v_i v_{i+1} v_i & = v_{i+1} v_i v_{i+1} & \text{ for } & i=1,\ldots, n-2; \label{rel-vvv}\\
        \sigma_i v_j & = v_j \sigma_i &  \text{ for } & |i-j| >1  ; \label{rel-height-sv}\\
        v_i \sigma_{i+1} v_i & = v_{i+1} \sigma_i v_{i+1} & \text{ for } & i=1,\ldots, n-2; \label{rel-vsv}\\
        \gamma_i^2 & = e & \text{ for } & i=1,\ldots, n; \label{rel-inverse-b}\\  
        \gamma_i \gamma_j & = \gamma_j \gamma_i   & \text{ for } & i,j=1,\ldots, n; \label{rel-height-bb} \\
        \gamma_j v_i & = v_i \gamma_j & \text{ for } & j\neq i, i+1; \label{rel-height-bv}\\
        \sigma_i\gamma_j & = \gamma_j\sigma_i & \text{ for } & j\neq i, i+1; \label{rel-height-sb}\\
        \gamma_{i+1} v_i & = v_{i} \gamma_i & \text{ for } & i=1,\ldots, n-1; \label{rel-bv} \\
        v_{i} \sigma_i v_{i} & = \gamma_{i+1} \gamma_i \sigma_{i} \gamma_i \gamma_{i+1} & \text{ for } &  i=1,\ldots, n-1. \label{rel-twist-III}
    \end{align}
\end{theorem}

\begin{remark} 
Using $(\ref{rel-inverse-v})$, we see that relations $(\ref{rel-vsv})$ and $(\ref{rel-bv})$ are equivalent the following $(\ref{relC-vsv})$ and $(\ref{relC-vb})$, respectively: 
    \begin{align}
      \sigma_{i+1} & = v_i v_{i+1} \sigma_i v_{i+1} v_i & \text{ for } & i=1,\ldots, n-2, \label{relC-vsv} \\ 
      \gamma_{i+1} & = v_i \gamma_i v_i & \text{ for } & i=1,\ldots, n-1.  \label{relC-vb}
     \end{align} 
\end{remark}

\begin{remark} 
There are two kinds of twisted braid moves of type ${\rm I}$ as shown in Figure~\ref{moves}. The left one corresponds to relations $(\ref{rel-bv})$ and the right one to $(\ref{rel-vb})$: 
    \begin{align}
       \gamma_i v_i  & = v_i \gamma_{i+1}  & \text{ for } & i=1,\ldots, n-1.  \label{rel-vb}
     \end{align}
Using $(\ref{rel-inverse-v})$, we see that relations $(\ref{rel-bv})$ are equivalent to $(\ref{rel-vb})$. \end{remark}

\begin{remark}
There are two kinds of twisted braid moves of ${\rm III}$; one is type ${\rm III(+)}$ as shown in Figure~\ref{moves}
and the other is type ${\rm III(-)}$.  The former 
corresponds to relations $(\ref{rel-twist-III})$ 
and the latter to $(\ref{rel-twist-III-negative})$: 
    \begin{align}
v_{i} \sigma_i^{-1} v_{i} & = \gamma_{i+1} \gamma_i \sigma_{i}^{-1} \gamma_i \gamma_{i+1} 
& \text{ for } &  i=1,\ldots, n-1.      \label{rel-twist-III-negative}
    \end{align}
Using $(\ref{rel-inverse-v})$ and $(\ref{rel-inverse-b})$, we see that relations $(\ref{rel-twist-III})$ are equivalent to $(\ref{rel-twist-III-negative})$.    
\end{remark}

\begin{proof} 
Note that the inverse elements of $v_i$ and $\gamma_i$ in $TVB_n$ are themselves. 
Let $\mathcal{S}$ be the set of standard generators of $TVB_n$ and let ${\mathcal{S}}^*$ be the set of standard generators and their inverse elements of $TVB_n$: 
    \begin{align*}
{\mathcal{S}}    &= \{ \sigma_i, v_i \mid i=1, \dots, n-1 \} \cup \{ \gamma_i \mid i=1, \dots, n\}, \\ 
{\mathcal{S}}^*  &= \{ \sigma_i, \sigma_i^{-1}, v_i \mid i=1, \dots, n-1 \} \cup \{ \gamma_i \mid i=1, \dots, n\}.
    \end{align*}
Let $b$ be a twisted virtual braid diagram. 
When it is good, it is presented uniquely as a concatenation of elements of ${\mathcal{S}}^*$, which we call a {\it preferred word} of $b$. 
When it is not good, one can modify it slightly by an isotopy of $E$ keeping the condition of a twisted virtual braid diagram to become good.
Thus, ${\mathcal{S}}$ generates the group $TVB_n$.  

Let $b$ and $b'$ are good twisted virtual braid diagrams.
Suppose that $b'$ is obtained from $b$ by an isotopy of $E$ keeping the condition of a twisted virtual braid diagram.
Then they are related by a finite sequence of changing heights of a pair of points in $X(b)$.
A single height change of a pair of such points corresponds to one of relations 
(\ref{rel-height-ss}), (\ref{rel-height-vv}),   (\ref{rel-height-sv}), (\ref{rel-height-bb}),  (\ref{rel-height-bv}), 
 (\ref{rel-height-sb}) and variants of (\ref{rel-height-ss}),   (\ref{rel-height-sv}) and  (\ref{rel-height-sb}) 
 with $\sigma_i$ replaced by $\sigma_i^{-1}$ and/or $\sigma_j$ replaced by $\sigma_j^{-1}$.  Note that the variants are consequences of the original relations up to relations (\ref{rel-inverse-v}) and (\ref{rel-inverse-b}). 
Thus, we see that the preferred words of $b$ and $b'$ are congruent modulo 
  relations  (\ref{rel-height-ss}), (\ref{rel-height-vv}),   (\ref{rel-height-sv}), (\ref{rel-height-bb}),  (\ref{rel-height-bv}), 
 (\ref{rel-height-sb}) and  relations (\ref{rel-inverse-v}) and (\ref{rel-inverse-b}).  
 
Suppose that $b'$ is obtained from $b$ by an extended Reidemeister move.
When the move is R2, the change of preferred words corresponds to 
$\sigma_i^{\epsilon} \sigma_i^{-\epsilon} = \sigma_i^{-\epsilon} \sigma_i^{\epsilon}$ $(\epsilon \in \{\pm 1\})$, which is a trivial relation.
When the move is R3, it is well known that the change of preferred words corresponds to 
a relation which is a consequence of relations (\ref{rel-sss}).
When the move is V2, the change of preferred words corresponds to relations (\ref{rel-inverse-v}).
When the move is V3, the change of preferred words corresponds to relations (\ref{rel-vvv}).
When the move is V4, we may assume that it is the move as in Figure~\ref{bmoves}, which 
corresponds to relations (\ref{rel-vsv}).
When the move is T1, the change of preferred words corresponds to relations (\ref{rel-bv}) or (\ref{rel-vb}).
When the move is T3, the change of preferred words corresponds to relations (\ref{rel-twist-III}) or (\ref{rel-twist-III-negative}). 
Therefore we see that the preferred words of $b$ and $b'$ are congruent each other modulo all relations 
(\ref{rel-height-ss})--(\ref{rel-twist-III}).  

Since all relations (\ref{rel-height-ss})--(\ref{rel-twist-III}) are valid in the group $TVB_n$, these relations are defining relations.  
\end{proof}

\begin{remark}
The twisted virtual group $TVB_n$ is different from the ring group (\cite{BH}) or the extended welded braid group (\cite{Damiani2017B}).  
Brendle and Hatcher \cite {BH} discussed the space of configurations of $n$ unlinked Euclidean circles, called {\it rings}, whose fundamental group is the {\it ring group} $R_n$.
They showed that the ring group is isomorphic to the {\it motion group} of the trivial link of $n$ components in the sense of Dahm \cite{Dahm}.
The ring group has a finite index subgroup isomorphic to the {\it braid-permutation group}, also called the {\it welded braid group}, introduced by Fenn, Rim{\' a}nyi and Rourke \cite{FRR}.
Damiani~\cite{Damiani2017A} studied the ring group from various points of view. 
In particular, she introduced in \cite {Damiani2017B} the notion of the {\it extended welded braid group} defined by using diagrams motivated from the work of Satoh~\cite{Satoh}.
Damiani's extended welded braid group is isomorphic to the ring group.  

The twisted virtual braid group $TVB_n$ is different from the ring group and the extended welded braid group for $n>2$, since they admit a relation $v_1 \sigma_2 \sigma_1 = \sigma_2 \sigma_1 v_2$, which is not allowed in the twisted virtual braid group $TVB_n$.  
\end{remark}


\section{Braid presentation of twisted links}
\label{sect:Alexander}

The closure of a twisted virtual braid (diagram) is defined by a similar way for a classical braid.

\begin{Example}
The closure of a twisted virtual braid diagram is shown in Figure~\ref{one}.
\begin{figure}[ht]
  \centering
    \includegraphics[width=2.5cm,height=3cm]{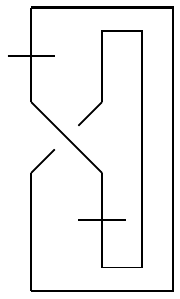}
        \caption{The closure of braid $\gamma_1\sigma_1^{-1}\gamma_2$.}
        \label{one}
        \end{figure}
\end{Example}

In this section we show that every twisted link is represented by the closure of a twisted virtual braid diagram (Theorem~\ref{prop:AlexanderB}).  

\subsection{Gauss data}

For a twisted link diagram $K$, we prepare some notation:
\begin{itemize}
    \item Let $V_R(K)$ be the set of all real crossings of $K$.
    \item Let $S(K)$ be the map from $V_R(K)$ to the set $\{+1,-1 \}$ assigning the signs to real crossings.
    \item Let $B(K)$ be the set of all bars in $K$.
    \item Let $N(v)$ be a regular neighborhood of $v$, where $v \in V_R(K) \cup B(K)$.
    \item For $c \in V_R(K)$, we denote by $c^{(1)}, c^{(2)}, c^{(3)}$, and $c^{(4)}$ the four points of $\partial N(c) \cap K$  as depicted in Figure~\ref{r}. 
    \begin{figure}[h]
  \centering
    \includegraphics[width=9cm,height=3cm]{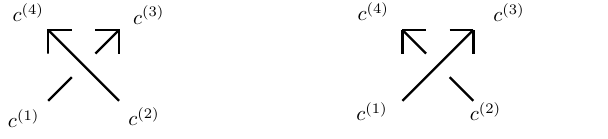}
        \caption{Boundary points of $N(c) \cap K$.}
        \label{r}
        \end{figure}
     \item For $\gamma \in B(K)$, we denote by $\gamma^{(1)}$ and $\gamma^{(2)}$ the two points of $\partial N(\gamma) \cap K$ as depicted in Figure \ref{b}.
      \begin{figure}[h]
  \centering
    \includegraphics[width=1cm,height=3cm]{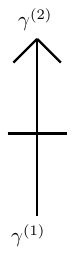}
        \caption{Boundary points of $N(\gamma) \cap K$.}
        \label{b}
        \end{figure}
     \item Put $W=W(K)=Cl(\mathbb{R}^2 \setminus \cup_{v\in V_R(K) \cup B(K)} N(v))$, where $Cl$ means the closure.
     \item Let $V^{\partial}_R(K) = \{c^{(j)} | c \in V_R(K), 1\leq j \leq 4\}$, and $B^{\partial}(K) = \{\gamma^{(j)} | \gamma \in B(K), 1\leq j \leq 2\}.$ 
     \item Let $K |_W$ be the restriction of $K$ to $W$, which is a union of some oriented arcs and loops generically immersed in $W$ such that the double points are virtual crossings of $K$, and the set of boundary points of the arcs is  the set $V^{\partial}_R(K) \cup B^{\partial}(K).$
     \item Let $\mu(K)$ be the number of components of $K$.
     \item Define a subset $G(K)$ of $(V^{\partial}_R(K) \cup B^{\partial}(K))^2$ such that $(a,b) \in G(K)$ if and only if $K |_W$ has an oriented arc starting from $a$ and ends at $b$.
\end{itemize}

The {\it Gauss data} of a twisted link diagram $K$ is the quintuple 
$$(V_R(K), S(K), B(K), G(K), \mu(K)).$$ 

\begin{Example}
Let $K$ be a twisted link diagram depicted in Figure~\ref{gd}.  When we name the real crossings $c_1$ and $c_2$ as in the figure, the Gauss data is  
$$(\{c_1,c_2\}, \{+1,+1\}, \{\gamma_1\},\{ (c_1^{(4)}, c_2^{(2)}), (c_2^{(3)}, c_1^{(2)}), (c_2^{(4)}, \gamma_1^{(1)}), (\gamma_1^{(2)}, c_1^{(1)}), (c_1^{(3)}, c_2^{(1)}) \},1).$$
 
\begin{figure}[h]
  \centering
    \includegraphics[width=7cm,height=4.5cm]{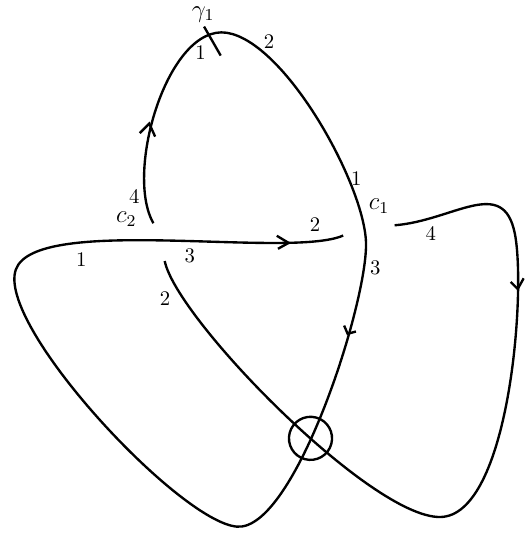}
        \caption{A twisted link diagram with one bar.}
        \label{gd}
        \end{figure}
\end{Example}

We say that two twisted link diagrams $K$ and $K'$ have {\it the same Gauss data} if 
$\mu (K) = \mu (K')$ and there exists a bijection $g:V_R(K) \cup B(K) \to V_R(K) \cup B(K)$ satisfying 
the following conditions:
\begin{itemize}
    \item $g(V_R(K)) = V_R(K)$, and $g(B(K))= B(K)$.   
    \item $g$ preserves the signs of real crossings; $S(K)(c) =S(K')(g(c))$ for $c \in V_R(K)$.  
    \item $(a,b) \in G(K)$ if and only if $(g^{\partial}(a), g^{\partial}(b)) \in G(K')$, where  $g^{\partial}: V^{\partial}_R(K) \cup B^{\partial}(K) \to V^{\partial}_R(K') \cup B^{\partial}(K')$ is the bijection induced from $g$, i.e., $g^{\partial}(c^{(j)}) = (g(c))^{(j)}$ for $c \in V_R(K), 1\leq j \leq 4$ and 
    $g^{\partial}(\gamma^{(j)}) = (g(\gamma))^{(j)}$ for $\gamma \in B(K), 1\leq j \leq 2$.
\end{itemize}

Let $K$ be a twisted link diagram and $W=W(K)=Cl(\mathbb{R}^2 \setminus \cup_{v\in V_R(K) \cup B(K)} N(v))$ as before. Suppose that $K'$ is a twisted link diagram with the same Gauss data with $K$. Then by an isotopy of $\mathbb{R}^2$ we can move $K'$ such that 
\begin{itemize}
    \item $K$ and $K'$ are identical in $N(v)$ for every $v \in V_R(K) \cup B(K)$,
    \item $K'$ has no real crossings and bars in $W$, and
    \item there is a bijection between the arcs/loops of $K|_W$ and those of $K'|_W$ with respect to the endpoints of the arcs.
\end{itemize}
In this situation, we say that $K'$ is obtained from $K$ {\it by replacing $K|_W$}.

\begin{lemma}\label{lemma:Same Gauss Data}
Let $K$ and $K'$ be twisted link diagrams, and let $W=W(K)=Cl(\mathbb{R}^2 \setminus \cup_{v\in V_R(K) \cup B(K)} N(v))$. 

  (1) If $K'$ is obtained from $K$ by replacing $K|_W$, then they are related by a finite sequence of isotopies of $\mathbb{R}^2$ with support $W$ and V1, V2, V3, V4, and T1 moves.  

  (2) If two twisted link diagrams $K$ and $K'$ have the same Gauss data, then $K$ is equivalent to $K'$.
\end{lemma}

\begin{proof}
(1) Let $N_1, N_2, \ldots, N_m $ be regular neighborhoods of the real crossings and bars of $K$. 
Let $a_1, a_2, \ldots, a_n $ and $a'_1, a'_2, \ldots, a'_n $ be the arcs/loops of $K|_W$ and $K'|_W$ respectively.  
Using an isotopy of $\mathbb{R}^2$ with support $W$, we may assume that the intersection of $a'_1$ with $a_2, \ldots, a_n $ are transverse double points.
The arc/loop $a_1$ is homotopic to $a'_1$ in $\mathbb{R}^2$ (relative to the boundary when $a_1$ is an arc).
Taking the homotopy generically with respect to the arcs/loops $ a_2, \ldots, a_n $, and the 2-disks $N_1, N_2, \ldots, N_m $, we see that the arc/loop $a_1$ can be changed into $a'_1$ by a finite sequence of moves as shown in Figure~\ref{mv} up to isotopy of $\mathbb{R}^2$  with support $W$.
Considering that all crossings in Figure~\ref{mv} are virtual crossings, we regard these moves as V1, V2, V3, V4, and T1 moves.  
In this way, we can change $a_1$ into $a_1'$ without changing other arcs/loops of $K|_W$ and $K'|_W$.  
Applying this argument inductively, all arcs/loops of $K|_W$ change into the corresponding ones of $K'|_W$. 

\begin{figure}[h]
  \centering
    \includegraphics[width=12cm,height=2cm]{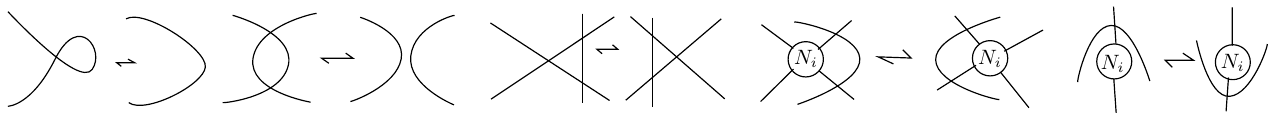}
        \caption{Moves on immersed curves.}
        \label{mv}
        \end{figure}
        
(2) Moving $K$ by an isotopy of $\mathbb{R}^2$, we may assume that $K'$ is obtained from $K$ by replacing $K|_W$.
By (1), we obtain the assertion. 
\end{proof}


\subsection{Braiding process}

Let $O$ be the origin of $\mathbb{R}^2$ and 
identify $\mathbb{R}^2 \setminus \{O\}$ with $\mathbb{R}_+\times S^1$ by polar coordinates, where $\mathbb{R}_+$ is the set of positive numbers. Let $\pi:\mathbb{R}^2 \setminus \{O\}=\mathbb{R}_+\times S^1 \to S^1$ denote the radial projection.  

For a twisted link diagram $K$, we denote by $V_R(K)$ the set of real crossings, by $V_B(K)$ the set of points on $K$ where bars intersect with, and by $X(K)$ the set of all (real or virtual) crossings and the set of points on $K$ where bars intersect with.    

\begin{definition}
A {\it closed twisted virtual braid diagram} is a twisted link diagram $K$ 
satisfying the following conditions (1) and (2):
\begin{itemize}
    \item[(1)] $K$ is contained in $\mathbb{R}^2 \setminus \{O\}$.
    \item[(2)] Let $k: \sqcup S^1 \to \mathbb{R}^2 \setminus \{O\}$ be the underlying immersion of $K$, where 
    $\sqcup S^1$ is a disjoint union of copies of $S^1$. 
    Then $\pi \circ k: \sqcup S^1 \to S^1$ is a covering map of $S^1$ of degree $n$ which respects the orientations of $\sqcup S^1$ and $S^1$.  
\end{itemize}
A closed twisted virtual braid diagram is {\it good} if it satisfies the following condition. 
\begin{itemize}
    \item[(3)] Let $N_1, N_2, \ldots, N_m $ be regular neighborhoods of the real crossings and bars of $K$. Then $\pi(N_i) \cap \pi(N_j) = \emptyset$ for $i \neq j$.  
\end{itemize}
\end{definition}

\begin{prop}\label{prop:AlexanderA}
Every twisted link diagram $K$ is equivalent, as a twisted link, to a good closed twisted virtual braid diagram $K'$ such that $K$ and $K'$ have the same Gauss data.  
\end{prop}

\begin{proof}
Let $K$ be a twisted link diagram and let $N_1, N_2, \ldots, N_m $ be regular neighborhoods of the real crossings and bars of $K$.
Moving $K$ by an isotopy of $\mathbb{R}^2$, we may assume that all $N_i$ are in $\mathbb{R}^2 \setminus \{O\}$,  $\pi(N_i) \cap \pi(N_j) =\emptyset$ for $i\neq j$ and the restriction of $K$ to $N_i$ satisfies the condition of a closed twisted virtual braid diagram.
Replace the remainder $K|_{W(K)}$ such that the result is a good closed twisted virtual braid diagram $K'$.
Then $K$ and $K'$ have the same Gauss data, and by Lemma~\ref{lemma:Same Gauss Data} they are equivalent as twisted links.  
\end{proof}

The procedure in the proof of Proposition~\ref{prop:AlexanderA} makes a given twisted link diagram to a good closed twisted virtual braid diagram having the same Gauss data with $K$. This is the {\it braiding process} in our paper. 

A point $\theta$ of $S^1$ is called a {\it regular value} for a closed twisted virtual braid diagram $K$ if $X(K)\cap\pi^{-1}(\theta)=\emptyset$.
Cutting $K$ along the half line $\pi^{-1}(\theta)$ for a regular value of $\theta$, we obtain a twisted virtual braid diagram whose closure is equivalent to $K$.  

Thus, Proposition~\ref{prop:AlexanderA} implies the following.  

\begin{theorem}\label{prop:AlexanderB}
Every twisted link is represented by the closure of a twisted virtual braid diagram.
\end{theorem}


\section{The Markov theorem for twisted links}
\label{sect:Markov}

In this section we show a theorem on braid presentation of twisted links which is 
analogous to the Markov theorem for classical links.

A {\it twisted Markov move of type $0$} or a {\it TM0-move} is a replacement of a twisted virtual braid diagram $b$ with another $b'$ of the same degree such that $b$ and $b'$ are equivalent as twisted virtual braids, i.e., they represent the same element of the twisted virtual braid group.   

A {\it twisted Markov move of type $1$} or a {\it TM1-move} is a replacement of a twisted virtual braid (or its diagram) $b$ with $b_1 b b_1^{-1}$ where  $b_1$ is a twisted virtual braid (or its diagram) of the same degree with $b$. We also call this move a {\it conjugation}.  

A {\it twisted Markov move of type $1$} or a {\it TM1-move} may be defined as a replacement of a twisted virtual braid (or its diagram) $b = b_1 b_2$ with $b' = b_2 b_1$ where  $b_1$ and $b_2$ are twisted virtual braids (or their diagrams) of the same degree.
See Figure~\ref{tm1}.
\begin{figure}[h]
  \centering
    \includegraphics[width=8cm,height=3cm]{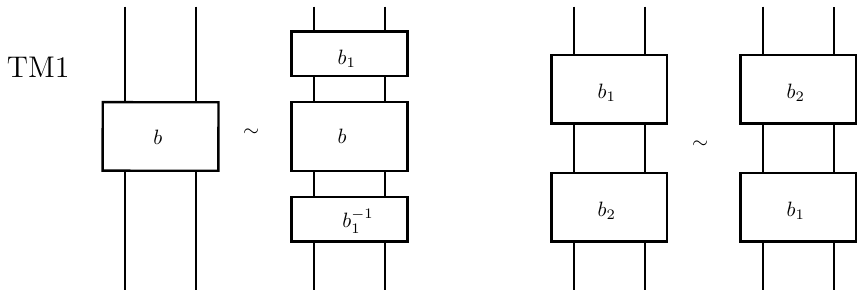}
        \caption{A twisted Markov move of type 1 or a TM1-move.}
        \label{tm1}
\end{figure}

For a twisted virtual braid (or its diagram) $b$ of degree $n$ and non-negative integers $s$ and $t$, we denote by $\iota_s^t(b)$ the twisted virtual braid (or its diagram) of degree $n + s + t$ obtained from $b$ by adding $s$ trivial strands to the left and $t$ trivial strands to the right.
This defines a monomorphism $\iota_s^t: TVB_n \to TVB_{n+s+t}$.  

A {\it stabilization of positive, negative or virtual type} is a replacement of a twisted virtual braid (or its diagram) $b$ of degree $n$ with $\iota_0^1(b)\sigma_n$, $\iota_0^1(b)\sigma_n^{-1}$ or $\iota_0^1(b)v_n$, respectively.
\begin{figure}[h]
  \centering
    \includegraphics[width=10cm,height=2cm]{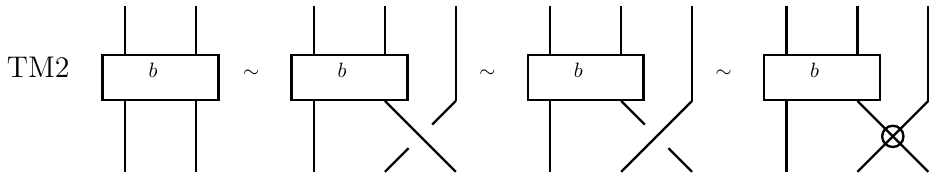}
        \caption{A twisted Markov move of type 2 or a TM2-move.}
        \label{tm2}
\end{figure}

A {\it twisted Markov move of type $2$} or a {\it TM2-move} is a  stabilization of positive, negative or virtual type, or its inverse operation.  
See Figure~\ref{tm2}.

A {\it right virtual exchange move} is a replacement 
$$ \iota_0^1(b_1) \sigma_n^{-1} \iota_0^1(b_2) \sigma_n \longleftrightarrow 
\iota_0^1(b_1) v_n \iota_0^1(b_2) v_n, $$ 
 and a {\it left virtual exchange move} is a replacement 
$$ \iota_1^0(b_1) \sigma_1^{-1} \iota_1^0(b_2) \sigma_1 \longleftrightarrow   
\iota_1^0(b_1) v_1 \iota_1^0(b_2) v_1, $$ 
where $b_1$ and $b_2$ are twisted virtual braids (or their diagrams).  
A {\it twisted Markov move of type $3$} or a {\it TM3-move} is a right/left virtual exchange move or its inverse operation. 
See Figure~\ref{tm3}.
\begin{figure}[h]
  \centering
    \includegraphics[width=10cm,height=3.5cm]{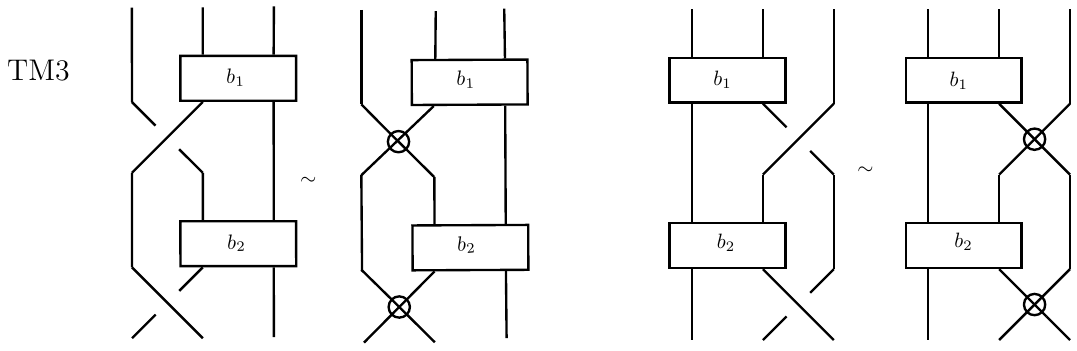}
        \caption{A twisted Markov move of type 3 or a TM3-move.}
        \label{tm3}
\end{figure}

\begin{definition}
Two twisted virtual braids (or their diagrams) are {\it Markov equivalent} if they are related by a finite sequence of twisted Markov moves TM1--TM3 (or TM0--TM3 when we discuss them as diagrams). 
\end{definition}

\begin{theorem}\label{theorem:MarkovA}
Two twisted virtual braids (or their diagrams) have equivalent closures as twisted links if and only if they are Markov equivalent.
\end{theorem}

\begin{remark}
In Section~\ref{sect:exchange}, 
it turns out that if two twisted virtual braids (or their diagrams) is related by a left virtual exchange move then they are related by a sequence of TM1-moves (or TM0-moves and TM1-moves when we discuss them as diagrams) and a right virtual exchange move. Thus we may remove left virtual exchange moves from the definition of Markov equivalence. 
\end{remark}

Let $K$ and $K'$ be closed twisted virtual braid diagrams and let $b$ and $b'$ be twisted virtual braid diagrams obtained from $K$ and $K'$ by cutting along $\pi^{-1}(\theta)$ and $\pi^{-1}(\theta')$ for some regular values $\theta$ and $\theta'$.    
We say that $K'$ is obtained from $K$ by a {\it twisted Markov move of type $0$} or a {\it TM0-move} if they are equivalent as closed twisted virtual braids.
Note that $K'$ is obtained from $K$ by a TM0-move if and only if $b$ and $b'$ are related by a finite sequence of TM0-moves and TM1-moves.  
We say that $K'$ is obtained from $K$ by a {\it twisted Markov move of type $2$} or a {\it TM2-move} if $b$ and $b'$ are related by a TM2-move and some   TM1-moves.
We say that $K'$ is obtained from $K$ by a {\it twisted Markov move of type $3$} or a {\it TM3-move} if $b$ and $b'$ are related by a TM3-move and some   TM1-moves.

\begin{definition}
Two closed twisted virtual braid diagrams $K$ and $K'$ are {\it Markov equivalent} if they are related by a finite sequence of TM0-, TM2- and TM3-moves.  
\end{definition}

\begin{prop}\label{prop:closure}
Two closed twisted virtual braid diagrams $K$ and $K'$ are Markov equivalent if and only if twisted virtual braid diagrams $b$ and $b'$ are Markov equivalent, where $b$ and $b'$ are obtained from $K$ and $K'$ by cutting along $\pi^{-1}(\theta)$ and $\pi^{-1}(\theta')$ for some regular values $\theta$ and $\theta'$.  
\end{prop}

\begin{proof}
For a given closed twisted virtual braid diagram $K$, $b$ is uniquely determined up to TM1-moves. Then the assertion is trivial by definition. 
\end{proof} 

By Proposition~\ref{prop:closure}, 
Theorem~\ref{theorem:MarkovA} is equivalent to the following theorem.  

\begin{theorem}\label{theorem:MarkovB}
Two closed twisted virtual braid diagrams are equivalent as twisted links if and only if they are Markov equivalent.
\end{theorem}

To prove Theorem~\ref{theorem:MarkovB}, we require the following lemma.

\begin{lemma}\label{lem:unique}
Two closed twisted virtual braid diagrams with the same Gauss data are Markov equivalent.
\end{lemma}

\begin{proof}
Let $K$ and $K'$ be closed twisted virtual braids with the same Gauss data. Modifying them by isotopies of $\mathbb{R}^2 \setminus \{O\}$, we may assume that they are good.
Let $N_1, N_2, \ldots, N_m$ be regular neighborhoods of the real crossings and bars of $K$, and $N'_1, N'_2, \ldots, N'_m$ be regular neighborhoods of the corresponding real crossings and bars of $K'$. 

Case (I). Suppose that $\pi(N_1), \pi(N_2), \ldots, \pi(N_m)$ and $\pi(N'_1), \pi(N'_2), \ldots, \pi(N'_m)$ appear in $S^1$ in the same cyclic order.
Modifying $K$ by an isotopy of $\mathbb{R}^2 \setminus \{O\}$ keeping the condition of a good closed twisted virtual braid, 
we may assume that $N_1=N'_1, N_2=N'_2, \ldots, N_m=N'_m$ and the restrictions of $K$ and $K'$ to these disks are identical. 
Let $a_1, \dots, a_s$ be the arcs/loops of $K|_W$ and $a'_1, \dots, a'_s$ be the corresponding arcs/loops of $K'|_W$.  
Let $\theta \in S^1$ be a regular value for $K$ and $K'$ such that $\pi^{-1}(\theta)$ is disjoint from $N_1 \cup \dots \cup N_m$. 
If there exists an arc/loop $a_i$ of $K|_W$ such that $|a_i \cap \pi^{-1}(\theta)|\neq |a'_i \cap \pi^{-1}(\theta)|$, then move a small segment of $a_i$ or $a'_i$ toward the origin $O$ by some V2 moves which are 

TM0-moves and apply some VM2-moves of virtual type so that $|a_i \cap \pi^{-1}(\theta)|=|a'_i \cap \pi^{-1}(\theta)|$ after the modification.
Thus without loss of generality, we may assume that $|a_i \cap \pi^{-1}(\theta)|=|a'_i \cap \pi^{-1}(\theta)|$ for all $i=1, \dots, s$. 

Let $k : \sqcup S^1 \to \mathbb{R}^2 \setminus \{O\}$ and $k' : \sqcup S^1 \to \mathbb{R}^2 \setminus \{O\}$ be the underlying immersions of $K$ and $K'$, respectively, such that they are identical near the preimages of the real crossings and bars. 
Let $I_1, \ldots, I_s$ be arcs/loops in $\sqcup S^1$ with $k(I_i)=a_i$ and $k'(I_i)=a'_i$ for $i=1, \dots, s$.
Note that $\pi \circ k|_{I_i}$ and $\pi \circ k'|_{I_i}$ are orientation-preserving immersions into $S^1$ with $\pi \circ k|_{\partial I_i}=\pi \circ k'|_{\partial I_i}$.
Since $a_i$ and $a'_i$ have the same degree, so we have a homotopy $k_i^t: I_i \to \mathbb{R}^2 \setminus \{O\}$ $(t \in [0,1])$ of $I_i$ relative to the boundary $\partial I_i$ such that $k^0_i=k|_{I_i}$ and $k^1_i=k|_{I_i}$ and $\pi \circ k^t_i$ is an orientation-preserving immersion.
Taking such a homotopy generically with respect to the other arcs/loops of $K|_{W}$ and $K'|_{W}$ and the 2-disks $N_1, N_2, \ldots, N_m$, we see that $a_i$ can be transformed to $a'_i$ by a sequence of TM0-moves. 
Apply this procedure inductively, we can change  $a_1, \dots, a_s$ to $a'_1, \dots, a'_s$ by a sequence of TM0-moves and TM2-moves.  
Thus we see that $K$ is transformed into $K'$ by a finite sequence of TM0 and TM2-moves.

Case (II). Suppose that $\pi(N_1), \pi(N_2), \dots, \pi(N_m)$ and $\pi(N'_1), \pi(N'_2), \dots, \pi(N'_m)$ do not appear in $S^1$ in the same cyclic order.
It is sufficient to show that we can interchange the position of two consecutive $\pi(N_i)$'s. 
Suppose that we want to interchange $\pi(N_1)$ and $\pi(N_2)$. 

  (1) Suppose that $N_2$ is a neighborhood of a real crossing.  Figure~\ref{c} shows how to interchange $\pi(N_1)$ and $\pi(N_2)$ by TM0-moves and TM2-moves.  
  \begin{figure}[h]
  \centering
    \includegraphics[width=10cm,height=7cm]{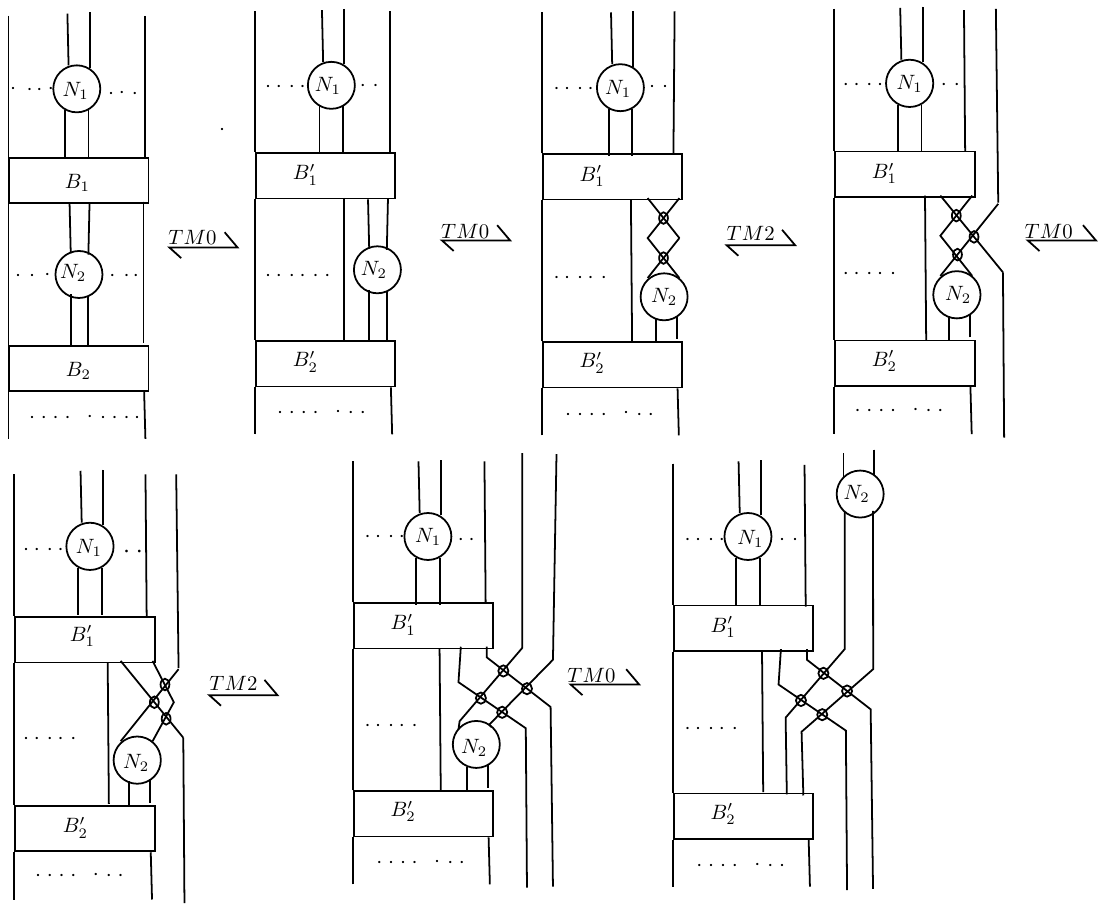}
        \caption{Interchange the positions of $N_1$ and $N_2$.}
        \label{c}
\end{figure}

  (2) Suppose that $N_2$ is a neighborhood of a bar.  Figure~\ref{c2} shows how to interchange $\pi(N_1)$ and $\pi(N_2)$ by TM0-moves and TM2-moves.  
\begin{figure}[h]
  \centering
    \includegraphics[width=10cm,height=4cm]{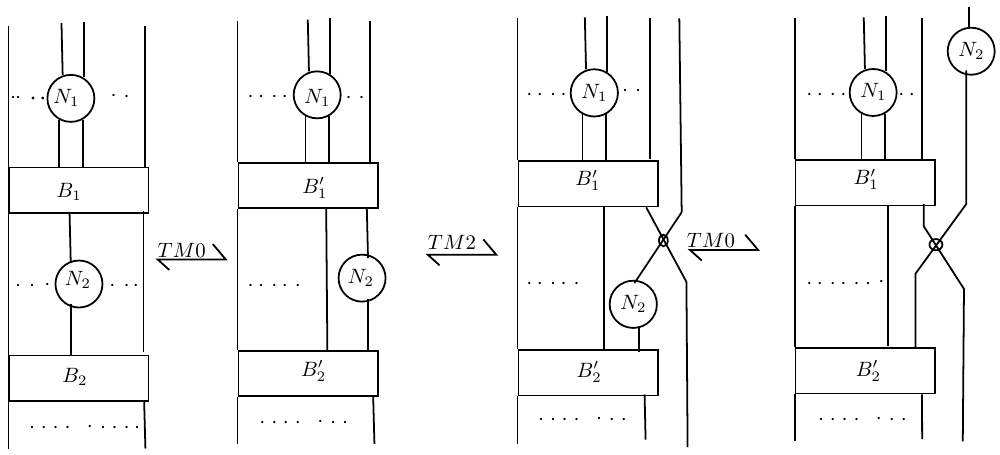}
        \caption{Interchange the positions of $N_1$ and $N_2$.}
        \label{c2}
\end{figure}

\sloppy{
Applying this argument, we can make $\pi(N_1), \pi(N_2), \ldots, \pi(N_m)$ and $\pi(N'_1), \pi(N'_2), \ldots, \pi(N'_m)$  to appear in the same cyclic order on $S^1$ using VM0 and VM2-moves.
 Then we can reduce the case to Case~(I). }
\end{proof}

\begin{proof}[Proof of Theorem~\ref{theorem:MarkovB}] 
If two closed twisted virtual braids (or their diagrams) are Markov equivalent then they are equivalent as twisted links.
Conversely, suppose that $K$ and $K'$ are closed twisted virtual braid diagrams which are equivalent as twisted links.  
There is a finite sequence of twisted link diagrams, say, $K=K_0, K_1, \ldots, K_n=K'$ such that $K_{i+1}$ is obtained from $K_{i}$ by one of the extended Reidemeister moves.  
 
For each $i = 1, \dots, n-1$, $K_i$ may not be a closed twisted virtual braid diagram.
Let $\widetilde K_i$ be a closed twisted virtual braid diagram obtained from $K_i$ by the braiding process in the previous section.
We assume $K_0=\widetilde K_0$ and $K_n =\widetilde K_n$. 
Then for each $i =0,1,\dots, n$, $\widetilde K_i$ and $K_i$ have the same Gauss data. 
It is sufficient to prove that $\widetilde K_i$ and $\widetilde K_{i+1}$ are Markov equivalent.
 
It is shown in \cite{sk} that  $\widetilde K_i$ and $\widetilde K_{i+1}$ are Markov equivalent when $K_{i+1}$ is obtained from $K_{i}$ by one of R1, R2, R3, V1, V2, V3, and V4.
(In \cite{sk} virtual links and closed virtual braid diagrams are discussed.
However the argument in \cite{sk} is valid in our current situation.)  

Thus, it is sufficient to consider a case that $K_{i+1}$ is obtained from $K_{i}$ by a twisted move T1, T2 or T3.  

(1)  Let $K_{i+1}$ be obtained by $K_i$ from a T1 move. 
Then $K_{i}$ and $K_{i+1}$ have same Gauss data, and hence $\widetilde K_{i}$ and $\widetilde K_{i+1}$ have same Gauss data.
By Lemma~\ref{lem:unique}, $\widetilde K_i$ and $\widetilde K_{i+1}$ are Markov equivalent.

(2) Let $K_{i+1}$ be obtained by $K_i$ by a T2 move. Assume that a pair of bars in $K_i$ is removed by the T2 move to obtain $K_{i+1}$.  
Let $N$ be a $2$-disk where the T2 move is applied such that $N \cap K_i$ is an arc, say $\alpha$, with two bars and $N \cap K_{i+1}$  is the arc $\alpha$. 
Let $N_1$ and $N_2$ be neighborhoods of the two bars such that $N_1 \cup N_2 \subset N$. 
By an isotopy of $\mathbb{R}^2$, deform $K_i$, $\alpha$ and $N$ such that $N \cap K_i$ is $\alpha$ with two bars and $\pi|_\alpha: \alpha \to S^1$ is an orientation-preserving embedding. 
Let $\widetilde K_i'$ be a closed twisted virtual braid obtained from the deformed $K_i$ by applying the braid process in the previous section such that $N$ is pointwise fixed, and let $\widetilde K_{i+1}'$ be a closed twisted virtual braid obtained from $\widetilde K_i'$ by removing the two bars intersecting $\alpha$.  
Then $\widetilde K_i'$ and $\widetilde K_{i+1}'$ are related by a TM0-move.   
Since $\widetilde K_i$ and $\widetilde K_i'$ have the same Gauss data, they are Markov equivalent. 
Since $\widetilde K_{i+1}$ and $\widetilde K_{i+1}'$ have the same Gauss data, they are Markov equivalent.
Thus $\widetilde K_i$ and $\widetilde K_{i+1}$ are Markov equivalent.  
The case that a pair of bars are introduced to $K_i$ to obtain $K_{i+1}$ is shown similarly.

(3) Let $K_{i+1}$ be obtained from $K_i$ by a T3 move.  
There are 4 possible orientations for a T3 move, say T3a, T3b, T3c, and T3d as in Figure~\ref{ot3m}. 
\begin{figure}[h]
  \centering
    \includegraphics[width=7cm,height=5cm]{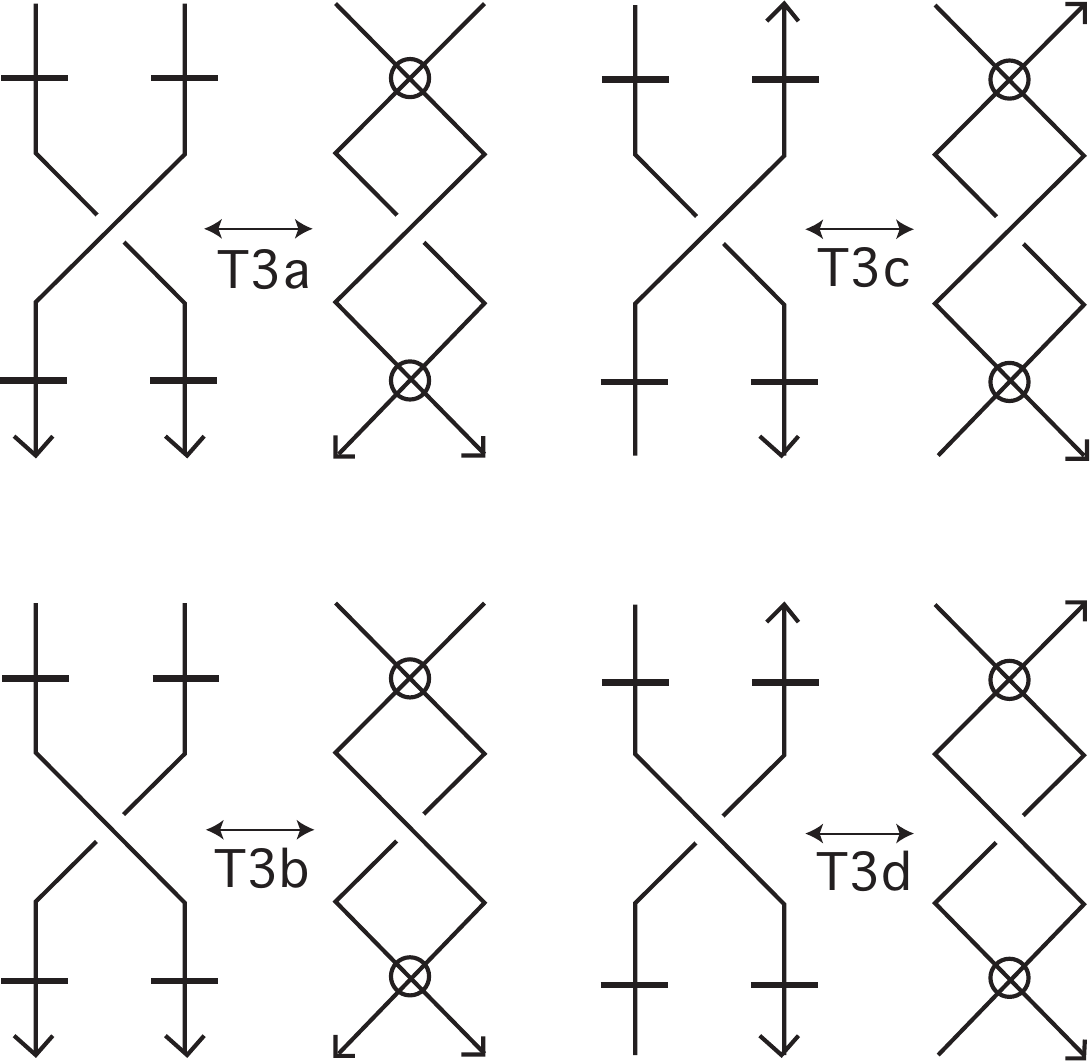}
        \caption{Oriented T3 moves.}
        \label{ot3m}
\end{figure}

First consider a case that $K_{i+1}$ is obtained from $K_i$ by a move T3a or T3b. 
Assume that $K_i$ is as in the left and $K_{i+1}$ is as in the right of Figure~\ref{ot3m}.  Let $N$ be a $2$-disk where the move is applied.  
Then $N \cap K_i$ is a pair of arcs, say $\alpha_1$ and $\alpha_2$, intersecting transversely at a real crossing and there are four bars. 
Let $N_1$ be a neighborhood of the real crossing of $K_i$ and $N_2, \dots, N_5$ be neighborhoods of the four bars of $K_i$ in $N$ such that $N_1 \cup \dots \cup N_5 \subset N$. 
By an isotopy of $\mathbb{R}^2$, deform $K_i$, $\alpha_1$, $\alpha_2$, and $N$ such that $\pi|_{\alpha_1}: \alpha_1 \to S^1$ and $\pi|_{\alpha_2}: \alpha_2 \to S^1$ are orientation-preserving embeddings. 
Let $\widetilde K_i'$ be a closed twisted virtual braid diagram obtained from the deformed $K_i$ by applying the braid process in the previous section such that $N$ is pointwise fixed, and let $\widetilde K_{i+1}'$ be a closed twisted virtual braid diagram obtained from $\widetilde K_i'$ by applying a T3a (or T3b) move.  
Then $\widetilde K_i'$ and $\widetilde K_{i+1}'$ are related by a TM0-move.   
Since $\widetilde K_i$ and $\widetilde K_i'$ have the same Gauss data, they are Markov equivalent. 
Since $\widetilde K_{i+1}$ and $\widetilde K_{i+1}'$ have the same Gauss data, they are Markov equivalent. Thus $\widetilde K_i$ and $\widetilde K_{i+1}$ are Markov equivalent.  
The case that $K_i$ is as in the right and $K_{i+1}$ is as in the left of the figure is shown similarly. 

Now consider the case that $K_{i+1}$ is obtained from $K_i$ by a move T3c or T3d. 
Note that a move T3c (or T3d) is a consequence of a move T3b (or T3a) modulo moves V1, V2, V3, and V4.  
One can see this by rotating the two diagrams in T3c (or T3d) by 90 degrees clockwise.
Then the left hand side becomes the same diagram with the left hand side of T3b (or T3a).
The right hand side of T3c (or T3d) after the rotation has a real crossing and no bars. One can see that the right hand side of T3b (or T3a) also has a real crossing and no bars. 
Considering the Gauss data of the tangle in $N$ and applying the same argument to the proof of Lemma~\ref{lem:unique}, we see that the right hand side of T3c (or T3d) after the rotation is transformed to the right hand side of T3b (or T3a) by  V1, V2, V3, and V4 moves in $N$. 
Thus we can reduce the case to T3a (or T3b) and the case of V1, V2, V3, and V4 moves.  
\end{proof}

\section{On virtual exchange moves of twisted virtual braids}
\label{sect:exchange}

It turns out that if two twisted virtual braids (or their diagrams) are related by a left virtual exchange move then they are related by a sequence of TM1-moves (or TM0-moves and TM1-moves) and a right virtual exchange move. Thus we may remove left virtual exchange moves from the definition of Markov equivalence. 

Let $f_n: TVB_n \to TVB_n$ be an isomorphism determined by 
\begin{align*}
\sigma_i & \mapsto \sigma_{n-i}, & \text{for } & i=1, \dots, n-1 \\ 
v_i & \mapsto v_{n-i}, & \text{for } & i=1, \dots, n-1 \\ 
\gamma_i & \mapsto \gamma_{n-i+1}, & \text{for } & i=1, \dots, n. 
\end{align*}

For a twisted virtual braid diagram $b$ of degree $n$ which is good,  
we also denote by $f_n(b)$ a twisted virtual braid diagram obtained from the diagram $b$ by applying the above correspondence to the preferred word of $b$.  

Let $\nabla_n$ be a twisted virtual braid (or its diagram) with  
\begin{align*}
\nabla_n  =  \prod_{i=1}^{n-1} (v_i v_{i-1} \dots v_1) \prod_{j=1}^{n} \gamma_j. 
\end{align*}

Let $F_n: TVB_n \to TVB_n$ be an isomorphism determined by 
\begin{align*}
b & \mapsto \nabla_n  b \nabla_n^{-1}  & \text{for } & b \in TVB_n. 
\end{align*}
Then $\nabla_n^2 = e$ in $TVB_n$ and $F_n(b) = f_n(b)$ for $b \in TVB_n$.  
In particular $b$ and $f_n(b)$ are related by a TM1-move (or TM0-moves and TM1-moves when we discuss them as diagrams).  

\begin{theorem} 
If two twisted virtual braids of degree $n$ (or their diagrams) are related by a left virtual exchange move, 
then they are related by a sequence of TM1-moves (or TM0-moves and TM1-moves) and a right virtual exchange move. 
\end{theorem}

\begin{proof}
Let $b$ and $b'$ be twisted virtual braid diagrams of degree $n$ related by a left virtual exchange move. 
Suppose that 
$$ b= \iota_1^0(b_1) \sigma_1^{-1} \iota_1^0(b_2) \sigma_1 \quad \mbox{and} \quad 
b'=    \iota_1^0(b_1) v_1 \iota_1^0(b_2) v_1, $$ 
where $b_1$ and $b_2$ are good twisted virtual braid diagrams of degree~$n-1$.  
Then 
$$ f_n(b) = \iota_0^1( f_{n-1}(b_1) ) \sigma_n^{-1} \iota_0^1( f_{n-1}(b_2) ) \sigma_n 
\quad \mbox{and} \quad 
f_n(b') =  \iota_0^1( f_{n-1}(b_1) ) v_n \iota_0^1( f_{n-1}(b_2) ) v_n, $$ 
and hence $f_n(b)$ and $f_n(b')$ are related by a right virtual exchange move.   
Since $b$ is conjugate to $F_n(b) = f_n(b)$ as elements of $TVB_n$, and 
$b'$ is conjugate to $F_n(b') = f_n(b')$, we see that $b$ and $b'$ are related by a sequence of TM1-moves 
(or TM0-moves and TM1-moves when we discuss them as diagrams) and a right virtual exchange move. 
\end{proof}


\section{A reduced presentation of the twisted virtual braid group}
\label{sect:reduced}

L. Kauffman and S. Lambropoulou~\cite{kl} gave a reduced presentation of the virtual braid group. Motivated by their work, we give a reduced presentation of the twisted virtual braid group.  
Using the reduced presentation, one can deal the twisted virtual braid group with less number of generators and relations.

In this section, we show that the presentation 
of the twisted virtual braid group $TVB_n$
given in Theorem~\ref{thm:StandardPresentation} can be reduced to a presentation with $n+1$ generators and less 
relations by rewriting $\sigma_i$ $(i=2,\ldots, n-1)$
and $\gamma_i$ $(i=2,\ldots, n)$ in terms of $\sigma_1$, $\gamma_1$ and $v_1, \dots, v_{n-1}$  as follows:
\begin{align}
    \sigma_i & =(v_{i-1}\ldots v_1)(v_i \ldots v_2)\sigma_1(v_2  \ldots v_i)(v_1  \ldots v_{i-1}) & \text{ for } & i=2,\ldots, n-1,  \label{1st reduction} \\ 
    \gamma_i & =(v_{i-1}\ldots v_1)\gamma_1(v_1  \ldots v_{i-1}) & \text{ for } & i=2,\ldots, n.  \label{2nd reduction}
\end{align}
See Figure~\ref{o}. These can be seen geometrically from their diagrams or algebraically from $(\ref{relC-vsv})$ and $(\ref{relC-vb})$.

\begin{figure}[h]
  \centering
    \includegraphics[width=6cm,height=10cm]{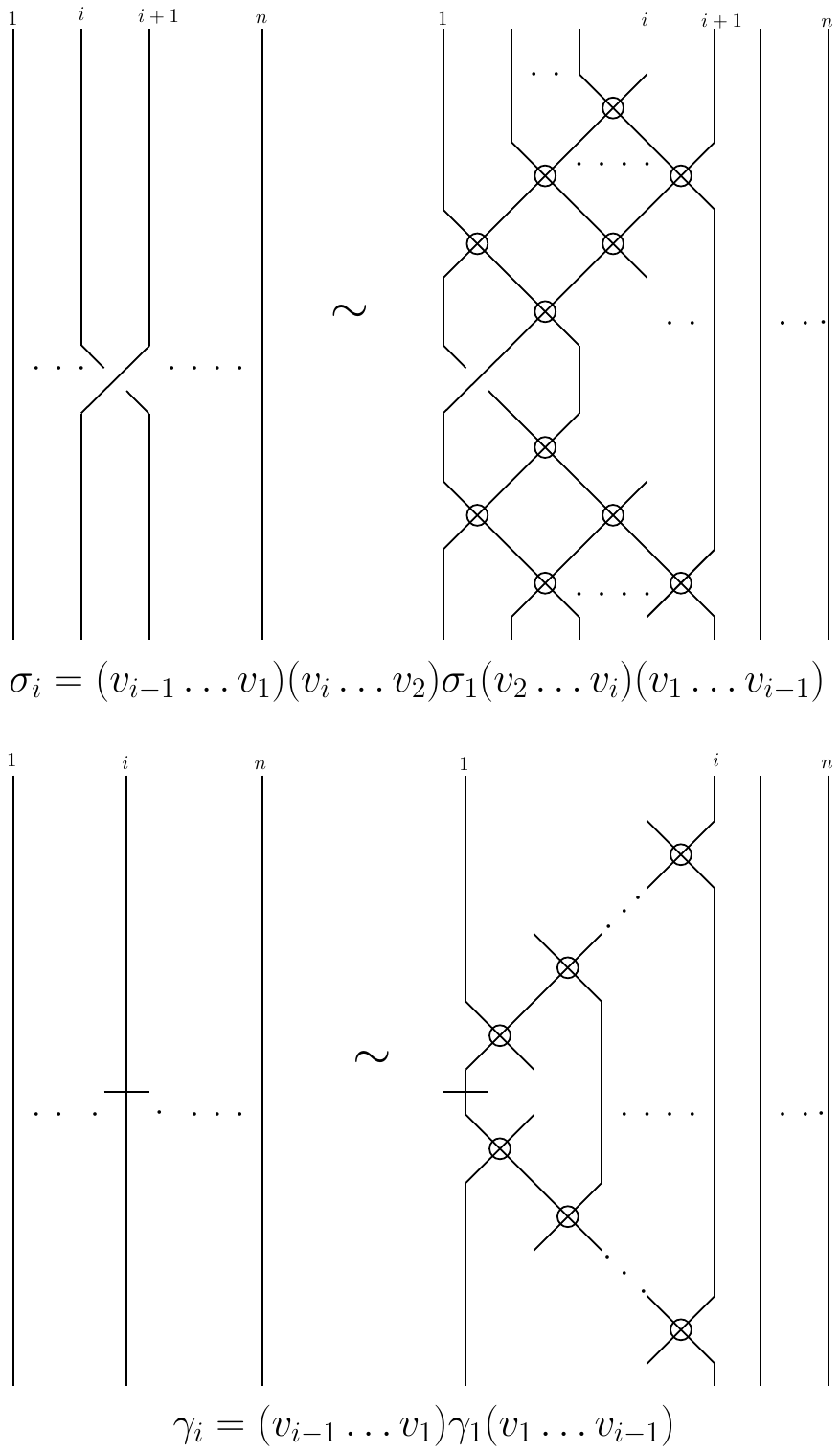}
        \caption{$\sigma_i \text{ and } \gamma_i$.}
        \label{o}
        \end{figure}

\begin{theorem}\label{thm:ReducedPresentation}
The twisted virtual braid group $TVB_n$ has a presentation whose generators are  
$\sigma_1, \gamma_1, v_1,\dots, v_{n-1}$   
and the defining relations are as follows:  
\begin{align}
 v_i^2  & = e   & \text{ for }  & i=1,\ldots, n-1;     \label{relB-inverse-v}\\
 v_iv_j  & = v_jv_i  & \text{ for } & |i-j| > 1 ;  \label{relB-height-vv}\\
 v_iv_{i+1}v_i & = v_{i+1}v_iv_{i+1} & \text{ for } & i=1,\ldots, n-2; \label{relB-vvv}\\
 \sigma_1(v_2v_3v_1v_2\sigma_1v_2v_1v_3v_2) & = (v_2v_3v_1v_2\sigma_1v_2v_1v_3v_2)\sigma_1, &  & \label{relB-height-ss} \\
  (v_1\sigma_1v_1)(v_2\sigma_{1}v_2)(v_1\sigma_1v_1) & = (v_2\sigma_1v_2)(v_1\sigma_{1}v_1)(v_2\sigma_1v_2), & & \label{relB-sss} \\
 \sigma_1v_j  & = v_j\sigma_1 & \text{ for } & j = 3, \ldots, n-1; \label{relB-height-sv}\\
 \gamma_1^2  & = e,  & &  \label{relB-inverse-b} \\
 \gamma_1v_j & =  v_j\gamma_1 & \text{ for } & j = 2, \ldots, n-1; & \label{relB-height-bv}\\
 \gamma_1v_1\gamma_1v_1 & =v_1\gamma_1v_1\gamma_1, & &  \label{relB-height-bb}\\
  \gamma_1v_1v_2\sigma_1v_2v_1 & = v_1v_2\sigma_1v_2v_1\gamma_1, & & \label{relB-height-sb}\\
  \gamma_{1}v_1\gamma_1\sigma_{1} \gamma_1v_1\gamma_{1} & = \sigma_1. & & \label{relB-bv}
  \end{align}
\end{theorem}

In what follows, we refer to relations $(\ref{rel-inverse-v})$, $(\ref{rel-height-vv})$ and $(\ref{rel-vvv})$ or 
equivalently $(\ref{relB-inverse-v})$, $(\ref{relB-height-vv})$ and $(\ref{relB-vvv})$ as the {\it virtual relations}.


\begin{lemma}[cf. \cite{kl}]
Relations $(\ref{rel-vsv})$ follow from relations $(\ref{1st reduction})$ and the virtual relations.  
\end{lemma}

This lemma is directly seen. The following three lemmas are proved in \cite{kl}. So we omit the proofs. 

\begin{lemma}[Lemma~1 of \cite{kl}]
Relations $(\ref{rel-height-sv})$ follow from relations $(\ref{1st reduction})$, the virtual relations, and  
relations $(\ref{relB-height-sv})$.  
\end{lemma}

\begin{lemma}[Lemma~3 of \cite{kl}] 
Relations $(\ref{rel-height-ss})$ follow from relations $(\ref{1st reduction})$, the virtual relations, and 
relations $(\ref{relB-height-ss})$ and $(\ref{relB-height-sv})$.  
\end{lemma}

\begin{lemma}[Lemma~2 of \cite{kl}] 
Relations $(\ref{rel-sss})$ follow from  relations $(\ref{1st reduction})$, the virtual relations, and 
relations $(\ref{relB-sss})$ and $(\ref{relB-height-sv})$.  
\end{lemma}

In the following proofs, we underline the expressions which we focus on.

\begin{lemma} 
Relations $(\ref{rel-inverse-b})$ follow from relations $(\ref{2nd reduction})$, the virtual relations, and 
relation $(\ref{relB-inverse-b})$. 
\end{lemma}

\begin{proof}
\begin{align*}
    \gamma_i^2 
           & = (v_{i-1}\ldots v_1)\gamma_1 \underline{(v_1 \ldots v_{i-1})(v_{i-1}\ldots v_1)}\gamma_1(v_1 \ldots v_{i-1})\\
           & = (v_{i-1}\ldots v_1)\underline{\gamma_1^2}(v_1 \ldots v_{i-1})\\
           & = \underline{(v_{i-1}\ldots v_1)(v_1 \ldots v_{i-1})}\\
           & = e.
\end{align*}
\end{proof}

\begin{lemma}
Relations $(\ref{rel-height-bv})$ follow from relations $(\ref{2nd reduction})$, the virtual relations, and relations $(\ref{relB-height-bv})$.
\end{lemma}
\begin{proof}
Since  $j \neq i, i + 1$, we consider the following two cases.

Case(i) Suppose $j\leq i-1$.  Then $i \geq 2$ and we have 
\begin{align*}
    v_i\gamma_j & = \underline{v_i}(v_{j-1}\ldots v_1)\gamma_1(v_1 \ldots v_{j-1})\\
           & = (v_{j-1}\ldots v_1)\underline{v_i\gamma_1}(v_1 \ldots v_{j-1})\\
           & = (v_{j-1}\ldots v_1)\gamma_1\underline{v_i}(v_1 \ldots v_{j-1})\\
           & = (v_{j-1}\ldots v_1)\gamma_1(v_1 \ldots v_{j-1})v_i\\
           & = \gamma_jv_i.
\end{align*}

Case(ii) Suppose $j\geq i+2$. Then 
\begin{align*}
    v_i\gamma_j & = \underline{v_i}(v_{j-1}\ldots v_1)\gamma_1(v_1 \ldots v_{j-1})\\
           & = (v_{j-1}\ldots v_{i+2}\underline{v_iv_{i+1}v_i}v_{i-1} \ldots v_1)\gamma_1(v_1 \ldots v_{j-1})\\
           & = (v_{j-1}\ldots v_{i+2}v_{i+1}v_i\underline{v_{i+1}}v_{i-1} \ldots v_1)\gamma_1(v_1 \ldots v_{j-1})\\
           & = (v_{j-1}\ldots v_{i+2}v_{i+1}v_iv_{i-1} \ldots v_1)\underline{v_{i+1}\gamma_1}(v_1 \ldots v_{j-1})\\
           & = (v_{j-1}\ldots v_1)\gamma_1\underline{v_{i+1}}(v_1 \ldots v_{j-1})\\
           & = (v_{j-1}\ldots v_1)\gamma_1(v_1 \ldots v_{i-1}\underline{v_{i+1}v_iv_{i+1}}v_{i+2} \ldots v_{j-1})\\
           & = (v_{j-1}\ldots v_1)\gamma_1(v_1 \ldots v_{i-1}v_iv_{i+1}\underline{v_{i}}v_{i+2} \ldots v_{j-1})\\
           & = (v_{j-1}\ldots v_1)\gamma_1(v_1 \ldots v_{j-1})v_i\\
           & = \gamma_jv_i.
\end{align*}
\end{proof}

\begin{lemma}
Relations $(\ref{rel-height-bb})$ 
follow from relations $(\ref{2nd reduction})$, the virtual relations, and  relations $(\ref{relB-height-bv})$
and $(\ref{relB-height-bb})$.
\end{lemma}

\begin{proof} By the previous lemma, we may assume relations $(\ref{rel-height-bv})$.  
It is sufficient to consider a case of $j > i$. 
\begin{align*} 
    \gamma_i\gamma_j 
  & = (v_{i-1}\ldots v_1) \gamma_1 (\underline{v_1 \ldots v_{i-1}}) (\underline{v_{j-1} \ldots v_1}) \gamma_1 (v_1 \ldots v_{j-1})\\
  & = (v_{i-1}\ldots v_1) \underline{\gamma_1} (v_{j-1} \ldots v_1) (v_2 \ldots v_i)  \underline{\gamma_1}  (v_1 \ldots v_{j-1}) 
  ~~~~~~~~~~~ \text{ (by (\ref{rel-height-bv}))}  \\
 & = (v_{i-1}\ldots v_1)  (v_{j-1} \ldots v_2) \underline{\gamma_1 v_1 \gamma_1}  (v_2 \ldots v_i)   (v_1 \ldots v_{j-1}) 
  ~~~~~~~~~~~ \text{ (by (\ref{relB-height-bb}))}      \\
  & = (\underline{v_{i-1}\ldots v_1})  \underline{(v_{j-1} \ldots v_2) v_1} \gamma_1 v_1 \gamma_1 v_1 (v_2 \ldots v_i)   (v_1 \ldots v_{j-1}) \\
  & =  (v_{j-1} \ldots v_2 v_1) (v_i \ldots v_2)  \gamma_1 v_1 \gamma_1 v_1 (\underline{v_2 \ldots v_i})   (\underline{v_1 \ldots v_{j-1}})  \\
  & =  (v_{j-1} \ldots v_2 v_1) (v_i \ldots v_2)  \gamma_1 v_1 \gamma_1 \underline{v_1}  (\underline{v_1} \ldots v_{j-1})  (v_1 \ldots v_{i-1})   \\
  & =  (v_{j-1} \ldots v_2 v_1) (v_i \ldots v_2)  \underline{\gamma_1} v_1 \underline{\gamma_1}  (v_2 \ldots v_{j-1})  (v_1 \ldots v_{i-1})  
 ~~~~~~~~~~~ \text{ (by (\ref{rel-height-bv}))}   \\ 
  & =  (v_{j-1} \ldots v_2 v_1) \gamma_1 (\underline{v_i \ldots v_2})   \underline{v_1  (v_2 \ldots v_{j-1})} \gamma_1  (v_1 \ldots v_{i-1})   \\ 
  & =  (v_{j-1} \ldots v_2 v_1) \gamma_1 v_1  (v_2 \ldots v_{j-1})  (v_{i-1} \ldots v_1)    \gamma_1  (v_1 \ldots v_{i-1})   \\ 
           & = \gamma_j\gamma_i.
\end{align*}
\end{proof}

\begin{lemma}
Relations $(\ref{rel-height-sb})$ 
follow from relations $(\ref{1st reduction})$, $(\ref{2nd reduction})$, the virtual relations, relations 
$(\ref{relB-height-sv})$, 
$(\ref{relB-height-bv})$ and 
$(\ref{relB-height-sb})$.  
\end{lemma}

\begin{proof}
By previous lemmas, we may assume relations $(\ref{rel-height-sv})$ and $(\ref{rel-vsv})$ or equivalently $(\ref{relC-vsv})$.  

First we show $(\ref{rel-height-sb})$ when  $j=1$, i.e., 
$\sigma_i\gamma_1=\gamma_1\sigma_i$ for $i \neq 1$.  
We apply induction on $i$, with initial condition $i=2$.
The relation $\sigma_2\gamma_1=\gamma_1\sigma_2$ follows from $(\ref{1st reduction})$ and $(\ref{relB-height-sb})$. 

Assuming $\sigma_i\gamma_1=\gamma_1\sigma_i$, we obtain $\sigma_{i+1}\gamma_1=\gamma_1\sigma_{i+1}$ as follows: 
\begin{align*}
    \sigma_{i+1}\gamma_1 & = v_{i}v_{i+1}\sigma_{i}v_{i+1}\underline{v_{i}\gamma_1}\\
                    & = v_{i}v_{i+1}\sigma_{i}\underline{v_{i+1}\gamma_1}v_{i}\\
                    & = v_{i}v_{i+1}\underline{\sigma_{i}\gamma_1}v_{i+1}v_{i}\\
                    & = v_{i}\underline{v_{i+1}\gamma_1}\sigma_{i}v_{i+1}v_{i}\\
                    & = \underline{v_{i}\gamma_1}v_{i+1}\sigma_{i}v_{i+1}v_{i}\\
                    & = \gamma_1 v_{i}v_{i+1}\sigma_{i}v_{i+1}v_{i}\\
                    & = \gamma_1\sigma_{i+1}.
\end{align*}
Hence,\begin{equation}
    \sigma_i\gamma_1=\gamma_1\sigma_i \quad \text{ for } i \neq 1. \label{3rd reduction}
\end{equation}

Now, we show relations $(\ref{rel-height-sb})$: $\sigma_i\gamma_j = \gamma_j\sigma_i$ for $j\neq i, i+1$.

Case(i) Suppose $j\leq i-1$. Then 
\begin{align*}
    \sigma_i\gamma_j & = \underline{\sigma_i }(v_{j-1}\ldots v_1)\gamma_1(v_1 \ldots v_{j-1})\\
                & = (v_{j-1}\ldots v_1)\underline{\sigma_i \gamma_1}(v_1 \ldots v_{j-1}) ~~~~~~~~~~~ \text{ (by (\ref{3rd reduction}))}\\
                & = (v_{j-1}\ldots v_1)\gamma_1\underline{\sigma_i }(v_1 \ldots v_{j-1})\\
                & = (v_{j-1}\ldots v_1)\gamma_1(v_1 \ldots v_{j-1})\sigma_i\\
                & = \gamma_j\sigma_i.
\end{align*}

Case(ii) Suppose $j\geq i+2$. Then 
\begin{align*}
    \sigma_i\gamma_j & = \underline{\sigma_i }(v_{j-1}\ldots v_1)\gamma_1(v_1 \ldots v_{j-1})\\
                & = (v_{j-1}\ldots v_{i+2}) \underline{\sigma_i} (v_{i+1}\ldots v_1)\gamma_1(v_1 \ldots v_{j-1})\\ 
                & = (v_{j-1}\ldots v_{i+2}) v_{i+1} v_i \sigma_{i+1} \underline{v_i v_{i+1}} (\underline{v_{i+1} v_i} v_{i-1} \ldots v_1)\gamma_1(v_1 \ldots v_{j-1})\\                 
                & = (v_{j-1}\ldots v_{i})\underline{\sigma_{i+1} }(v_{i-1}\ldots v_1)\gamma_1(v_1 \ldots v_{j-1})\\
                & = (v_{j-1}\ldots v_{i})(v_{i-1}\ldots v_1)\underline{\sigma_{i+1}\gamma_1 }(v_1 \ldots v_{j-1})~~~~~~~~~~~~~~ \text{ (by (\ref{3rd reduction}))}\\
                & = (v_{j-1}\ldots v_{1}) \gamma_1 \underline{\sigma_{i+1}} (v_1 \ldots v_{j-1})\\
                & = (v_{j-1}\ldots v_{1}) \gamma_1 (v_1 \ldots v_{i-1}) \underline{\sigma_{i+1}} (v_i \ldots v_{j-1})\\
                & = (v_{j-1}\ldots v_{1}) \gamma_1 (v_1 \ldots v_{i-1}) v_i v_{i+1} \sigma_i \underline{v_{i+1} v_i} ( \underline{v_i v_{i+1}} v_{i+2} \ldots v_{j-1})\\
                & = (v_{j-1}\ldots v_{1}) \gamma_1 (v_1 \ldots v_{i+1}) \underline{\sigma_{i}} (v_{i+2} \ldots v_{j-1})\\
                & = (v_{j-1}\ldots v_{1}) \gamma_1 (v_1 \ldots v_{i+1})(v_{i+2} \ldots v_{j-1})\sigma_{i} \\
                & = \gamma_j\sigma_i.
\end{align*}
\end{proof}

\begin{lemma}
Relations $(\ref{rel-bv})$ 
follow from relations $(\ref{2nd reduction})$ and the virtual relations.  
\end{lemma}

\begin{proof}
\begin{align*}
   \gamma_{i+1} v_i 
& = (v_i \ldots v_1) \gamma_1 (v_1 \ldots \underline{v_i}) \underline{v_i} \\
& = v_i ( v_{i-1} \ldots v_1) \gamma_1 (v_1 \ldots v_{i-1}) \\
& = v_i \gamma_i. 
\end{align*}
\end{proof}

\begin{lemma}
Relations $(\ref{rel-twist-III})$ 
follow from  relations $(\ref{1st reduction})$, $(\ref{2nd reduction})$, the virtual relations, and  relations 
$(\ref{relB-height-bv})$ and $(\ref{relB-bv})$.   
\end{lemma}

\begin{proof}
\begin{align*}
    \gamma_{i+1}\gamma_{i}\sigma_i\gamma_{i}\gamma_{i+1} 
 & = (v_i\ldots v_1)\gamma_1(v_1\ldots v_i)(v_{i-1} \ldots v_1)\gamma_1 \underline{(v_1 \ldots v_{i-1})(v_{i-1} \ldots v_1)}(v_{i} \ldots v_2)\\
    & \sigma_1(v_2 \ldots v_{i}) \underline{(v_1 \ldots v_{i-1})(v_{i-1} \ldots v_1)} \gamma_1(v_1 \ldots v_{i-1})(v_i \ldots v_1)\gamma_1(v_1 \ldots v_i)\\
 & = (v_i\ldots v_1)\gamma_1\underline{(v_1\ldots v_{i-1} v_i v_{i-1} \ldots v_1)} \gamma_1(v_{i} \ldots v_2)\sigma_1(v_2 \ldots v_{i})\gamma_1\\
     & \underline{(v_1 \ldots v_{i-1} v_i v_{i-1} \ldots v_1)}\gamma_1(v_1 \ldots v_i)\\
 & = (v_i\ldots v_1)\underline{\gamma_1}(v_i\ldots v_2 v_1 v_2 \ldots v_i)\underline{\gamma_1}(v_{i} \ldots v_2)\sigma_1(v_2 \ldots v_{i}) \underline{\gamma_1}\\
     & (v_i \ldots v_2 v_1 v_2 \ldots v_i) \underline{\gamma_1}(v_1 \ldots v_i)\\
     & = (v_i\ldots v_1)(v_i\ldots v_{2})\gamma_1 v_1\underline{(v_{2} \ldots v_i)(v_{i} \ldots v_2)}\gamma_1\sigma_1\gamma_1\underline{(v_2 \ldots v_{i})(v_i \ldots v_{2})}\\
     & v_1 \gamma_1 (v_{2} \ldots v_i)(v_1 \ldots v_i)\\
     & = (v_i\ldots v_1)(v_i\ldots v_{2})\underline{\gamma_1 v_1\gamma_1\sigma_1\gamma_1v_1\gamma_1}(v_{2} \ldots v_i)(v_1 \ldots v_i)\\
     & = v_i\underline{(v_{i-1}\ldots v_1)(v_i\ldots v_{2})\sigma_1(v_{2} \ldots v_i)(v_1 \ldots v_{i-1})}v_i\\
     & = v_i\sigma_iv_i.
\end{align*}
\end{proof}

\begin{proof}[Proof of Theorem~\ref{thm:ReducedPresentation}]
In the twisted virtual braid group, it is verified that all relations $(\ref{1st reduction})$--$(\ref{relB-bv})$ are valid by a geometrical argument using diagrams or algebraic argument using the relations $(\ref{rel-height-ss})$--$(\ref{rel-twist-III})$. On the other hands, we see that the relations $(\ref{rel-height-ss})$--$(\ref{rel-twist-III})$ follow from the relations $(\ref{1st reduction})$--$(\ref{relB-bv})$ by the previous lemmas. 
\end{proof}


\section*{Concluding remarks}

In this paper we study twisted virtual braids and the twisted virtual braid group, and provide theorems for twisted links corresponding to the Alexander theorem and the Markov theorem. We also provide a group presentation and a reduced group presentation of the twisted virtual braid group. 
As future work, it will be interesting to study the pure twisted virtual braid group and construct invariants for twisted virtual braids and twisted links. 
For example, biquandles with structures related to twisted links introduced in \cite{ns} may be discussed by using twisted virtual braids. 


\section*{Acknowledgements}

K.~Negi would like to thank the University Grants Commission(UGC), India, for Research Fellowship with NTA Ref.No.191620008047. M.~Prabhakar acknowledges the support given by the Science and Engineering Board(SERB), Department of Science $\&$ Technology, Government of India under grant-in-aid Mathematical Research Impact Centric Support(MATRICS) with F.No.MTR/2021/00394. This work was partially supported by the FIST program of the Department of Science and Technology, Government of India, Reference No. SR/FST/MS-I/2018/22(C), and was supported by JSPS KAKENHI Grant Number JP19H01788.

%
%

\end{document}